\documentclass[conference]{IEEEtran}
\IEEEoverridecommandlockouts
\usepackage{cite}
\usepackage{lmodern}
\usepackage{amsmath}
\usepackage{graphicx}
\usepackage[colorlinks=true, allcolors=black]{hyperref}
\usepackage{amssymb}
\usepackage{amsfonts}
\usepackage{amsthm}
\usepackage{colortbl}
\usepackage{dsfont}
\usepackage{bm}
\usepackage{algorithm}
\usepackage{algorithmic}
\usepackage{dsfont}
\usepackage[table, dvipsnames]{xcolor}
\usepackage{tabularx}
\usepackage{lipsum}
\usepackage{mathtools}
\usepackage{listings}

\newcommand{\espE}{\mathbb{E} }
\newtheorem{theorem}{Theorem}
\newtheorem{lemma}[theorem]{Lemma}

\newtheorem{definition}[theorem]{Definition}

\newtheorem{assumption}[theorem]{Assumption}

\def\BibTeX{{\rm B\kern-.05em{\sc i\kern-.025em b}\kern-.08em
    T\kern-.1667em\lower.7ex\hbox{E}\kern-.125emX}}
\begin{document}

\title{Mean-Covariance Steering of a Linear Stochastic System with Input Delay and Additive Noise}

\author{\IEEEauthorblockN{Gabriel Velho}
\IEEEauthorblockA{\textit{\small Université Paris-Saclay,} \\
\textit{\small CentraleSupélec, CNRS} \\
\textit{\small Laboratoire des Signaux} \\ 
\textit{\small et Systèmes,} \\
\small Gif-sur-Yvette, France \\
\footnotesize gabriel.velho@centralesupelec.fr}
\and
\IEEEauthorblockN{Riccardo Bonalli}
\IEEEauthorblockA{\textit{\small Université Paris-Saclay,} \\
\textit{\small CentraleSupélec, CNRS} \\
\textit{\small Laboratoire des Signaux} \\ 
\textit{\small et Systèmes,} \\
\small Gif-sur-Yvette, France \\
\footnotesize riccardo.bonalli@centralesupelec.fr}
\and
\IEEEauthorblockN{Jean Auriol}
\IEEEauthorblockA{\textit{\small Université Paris-Saclay,} \\
\textit{\small CentraleSupélec, CNRS} \\
\textit{\small Laboratoire des Signaux} \\ 
\textit{\small et Systèmes,} \\
\small Gif-sur-Yvette, France \\
\footnotesize jean.auriol@centralesupelec.fr}
\and
\IEEEauthorblockN{Islam Boussaada}
\IEEEauthorblockA{\textit{\small Université Paris-Saclay,} \\
\textit{\small CentraleSupélec, CNRS} \\
\textit{\small Laboratoire des Signaux} \\ 
\textit{\small et Systèmes,} \\
\textit{\small Institut Polytechnique} \\
\textit{\small des Sciences Avancées,} \\
\small Gif-sur-Yvette, France \\
\footnotesize islam.boussaada@centralesupelec.fr}
}

\maketitle

\begin{abstract}
In this paper, we introduce a novel approach to solve the (mean-covariance) steering problem for a fairly general class of linear continuous-time stochastic systems subject to input delays. 
Specifically, we aim at steering delayed linear stochastic differential equations to a final desired random variable with given mean and covariance. 
We first establish a controllability result for these control systems, revealing the existence of a lower bound under which the covariance of the control system cannot be steered. 
This structural threshold covariance stems from a unique combined effect due to stochastic diffusions and delays. Next, we propose a numerically cheap approach to reach any neighbor of this threshold covariance in finite time. Via an optimal control-based strategy, we enhance the aforementioned approach to keep the system covariance small at will in the whole control horizon. 
Under some additional assumptions on the dynamics, we give theoretical guarantees on the efficiency of our method. 
Finally, 
numerical simulations are provided to ground our theoretical findings, showcasing the ability of our methods in optimally approaching the covariance threshold.
\end{abstract}


\section{Introduction}



Dynamical processes are often affected by disturbances stemming from various factors, such as imprecise measurements, parameter uncertainties, and external disturbances. 
Such disturbances may considerably alter the  dynamics, as in several applications such as automated vehicle steering, traffic network control, or building heat regulation. 
Therefore, effectively mitigating these uncertainties is crucial to establish the reliability and security of such controlled systems.
Stochastic Differential Equations (SDEs) provide broad and accurate modelization of a large class of uncertain systems. 
Stochastic control enables the effective design of stabilizing controllers for SDEs, which are also robust against random fluctuations. In particular, such robustness may be reliably achieved by seeking controllers that keep the state variance relatively small, see \cite{ chen_optimal_2016, chen_optimal_2016-1 ,liu_optimal_2023} 
, and references therein. 

Another crucial consideration in the modeling of dynamical systems is the integration of delays into the dynamics. Delays in the system state or the control input stem from various sources, including physical constraints or transmission times~\cite{niculescu_stability_1998}. When delays take large values, neglecting them may lead to critical stability issues, hindering robust control of the system dynamics.
For deterministic systems, the generation of predictive state models has been suggested to handle delays \cite{artstein_linear_1982, krstic_boundary_2008}. However, these methods often depend on prior knowledge of the system dynamics, a challenging limitation when dealing with stochastic systems. Certain methods have been proposed to handle systems with unknown perturbations \cite{lechappe_new_2015, alves_lima_newton-series-based_2022}. However, their implementation still requires some prior knowledge of the noise structure. Consequently, such approaches can hardly be adjusted to stabilize delayed SDEs.

All the aforementioned hindrances show the urgency in developing novel methods to efficiently and robustly control stochastic systems with delays. In this regard, some stabilization methods have been proposed \cite{shaikhet_lyapunov_2013,cacace_predictor-based_2021}. Yet, state covariance minimization, often key to mitigating uncertainty, is generally disregarded. To effectively compute strategies seeking minimal covariance, optimal control methods have been alternatively investigated. These approaches are supported by necessary conditions for optimality, which are however efficiently implementable only in specific settings \cite{oksendal_maximum_2001, chen_maximum_2010, meng_global_2021}. 
Notably, in Linear Quadratic (LQ) settings, i.e., linear dynamics and quadratic costs, 
conditions for optimality may be efficiently solved by 
seeking solutions to Riccati-type ordinary differential equations \cite{liang_solution_2018, jin_tracking_2019, wang_linear_2023}, a numerically cheap method. Nevertheless, to the best of our knowledge, as efficient as they may be, these approaches do not support final state constraints. Importantly, estimates bounding the state covariance are generally underrated and not investigated, although these are crucial to establish the system's overall safety \cite{wang_risk-averse_2022}.


 In this paper, by merging methodologies from deterministic delayed control and stochastic control, we start bridging these gaps. 
Specifically, for the first time, we introduce a novel approach to control linear SDEs with delays under guarantees ensuring the state covariance is kept small throughout the control horizon. Our controls are uniquely cheap to numerically implement. 
Our contribution is threefold: 

\begin{enumerate}
    \item First, we investigate the controllability, in mean and covariance, of delayed SDEs. For this, we extend Arstein-type transformations \cite{manitius_finite_1979, artstein_linear_1982}. Unlike non-delayed settings, our analysis shows the covariance of linear delayed SDEs can not be steered to any symmetric definite positive matrix. Instead, due to the combined presence of diffusions and delays, the system can never be sterred beyond some \textit{minimal covariance}\footnote{with respect to the partial order in the space of symmetric definite matrices.}. This is a structural obstruction exclusive to delayed SDEs. Still, under classical rank-type controllability conditions, we prove an open-loop control exists enabling to steer delayed SDEs close to this minimal covariance at will.
    
    \item The implementation of open-loop controllers is notably impractical in the stochastic framework. To achieve the design of numerically tractable control laws, we leverage covariance steering techniques \cite{chen_optimal_2016, liu_optimal_2023}. This will enable computing feedback controls that steer linear delayed SDEs from an initial covariance to any final covariance as close as wanted to the aforementioned minimal covariance while minimizing the control effort.
    
    \item The previous class of controllers does not guarantee the covariance remains small throughout the whole control horizon. To bridge this gap, we leverage optimal control techniques to minimize both the final state covariance and the covariance along the whole trajectory. Importantly, under some additional assumptions on the dynamics, we provide estimates of these minimal covariances in the autonomous case. These bounds show the covariance of delayed SDEs can be forced to evolve within any neighbor of the minimal covariance. 
    
    
\end{enumerate}


The paper is organized as follows. In Section \ref{PF}, we outline the problem formulation. Our first controllability result is stated and proved in Section \ref{CC}. In Section \ref{CS}, we introduce a numerically efficient methodology to steer the system covariance. Upon these results, in Section \ref{OC} we develop a technique to minimize the covariance throughout the control horizon, developing error estimates. Finally, in Section \ref{NR}, we present numerical simulations on a real-world system model for building temperature control, subject to realistic temperature transmission delays and random external temperature fluctuations.

\section{Problem formulation}\label{PF}

\subsection{Notations} Let us consider $n$ and $m$ two positive integers. We assume state variables take values in $\mathbb{R}^n$, while  control variables take values in $\mathbb{R}^m$. We assume we are given a filtered probability space $(\Omega, \mathcal{F} \triangleq (\mathcal{F}_t)_{t \in [0,\infty)}, \mathbb{P})$. For the sake of clarity in the exposition and without loss of generality, from now on, we assume stochastic perturbations are due to a one-dimensional Wiener process $W_t$, which is adapted to the filtration $\mathcal{F}$. 
Let $T > 0$ be some given time horizon, while $0 < h < T$ is some fixed delay. For any $r \in \mathbb{N}$, we denote by $L_\mathcal{F}^2([0,T] , \mathbb{R}^r)$ the set of square integrable processes $P: [0,T]\times\Omega \to \mathbb{R}^r$ that are adapted to $\mathcal{F}$. The spaces of semi-definite and definite positive symmetric matrices in $\mathbb{R}^n$ are denoted by $\mathcal{S}^+_n$ and $\mathcal{S}^{++}_n$, respectively. If $M(t)$ is an $L^\infty([0,T], \mathbb{R}^{n \times n}) $ matrix function, we denote $\Phi_M(t,s)$ the fundamental matrix associated to it, which is by definition the unique solution to the system 
\begin{equation*}
\left\{
    \begin{array}{ll}
    \displaystyle \frac{d \Phi_M}{dt}(t,s) & = M(t) \Phi_M(t,s), \\
    \Phi_M(s,s) & = I .
    \end{array}
    \right.
\end{equation*}
If $X \in L_\mathcal{F}^2([0,T] , \mathbb{R}^r)$, we denote by $\Sigma_X(\cdot)$ its covariance (matrix), which is defined as
$$
\Sigma_X(t) \triangleq \espE[ (X(t) - \espE[X(t)]) (X(t) - \espE[X(t)])^T ] \in \mathcal{S}^+_n .
$$

\subsection{Control system and assumptions}

In this paper, we consider delayed-input SDEs of the form
\begin{equation}\label{eq:base_system_linear_delayed_0}
\left\{
    \begin{array}{ll}
    dX(t) &= \hspace{1em} \left( A(t) X(t) + B(t) U(t - h) + r(t) \right) dt \\
    &\qquad + \sigma(t) dW_t, \\
    X(0) &= \hspace{1em} X_0 , \\
    U(s) &= \hspace{1em} 0 \quad \text{for } s \in [-h, 0[,
    \end{array}
    \right.
\end{equation}
where $X_0 \in \mathbb{R}^n$ is a fixed initial condition, 
whereas the control $U$ lies in the control space $\mathcal{U} \triangleq L_\mathcal{F}^2([0,T] , \mathbb{R}^m)$. The constant $h>0$ is a positive delay acting on the control input. 
Given $X_T \in \mathbb{R}^n$ and $\Sigma_T \in \mathcal{S}^{++}_n$, our goal is to find $U \in \mathcal{U}$ that steers the corresponding solution $X$ of \eqref{eq:base_system_linear_delayed_0} to
\begin{equation*}\label{eq:original_steering_problem}
\left( \begin{array}{cc} \espE[X(T)] \\ \Sigma_X(T) \end{array} \right) = \left( \begin{array}{cc} X_T \\ \Sigma_T \end{array} \right) . 
\end{equation*}
We call this problem the \textit{(mean-covariance) steering problem}. It boils down to finding the control that displaces the initial probability distribution of the state to a more desirable final distribution. To fulfil this control objective, we make the following assumption.

\begin{assumption}\label{asm:assumptions_on_the_dynamic}
Throughout the paper, we assume the following properties hold true: 
\begin{itemize}
    \item $A \in L^\infty([0,T] , \mathbb{R}^{n \times n}) $, $B \in L^\infty([0,T] , \mathbb{R}^{n \times m})$, $r \in L^{\infty}([0,T] , \mathbb{R}^{n})$ and $\sigma \in L^{\infty}([0,T] , \mathbb{R}^{n}) $. In particular, these mappings are deterministic.
    \item 
    The Grammian associated with the non-delayed deterministic system, defined as
    \begin{equation}\label{eq:definition_grammian}
            G_\tau^{T+h} \triangleq \int_\tau^{T+h} \hspace{-0.4cm}\Phi_A(T+h,s) B(s) B(s)^T \Phi_A(s, T+h)^T dt,
        \end{equation}
    is invertible for all $0 < \tau < T+h$. 
\end{itemize}
\end{assumption}
The first assumption on the dynamics is classical as it guarantees the well-posedness of the system \cite{yong_stochastic_1999}. The second assumption allows for the total controllability of the deterministic non-delayed system, meaning it can be controlled in any sub-interval of $[0,T+h]$ \cite{coron_control_2009}. Furthermore, we assume that the initial state $X_0$ is deterministic. While our approach can be extended straightforwardly to Gaussian random variables (as in \cite{chen_optimal_2016}), we maintain a null initial covariance for the sake of simplicity. The above assumptions are generic and usually satisfied by a wide class of systems. 

\subsection{System reduction}

We can reduce the steering problem to an easier one, where $X_0$ and $X_T$ are both 0 and where the drift $r(t)$ is the zero function. 
For this, 
let $X_r(t) \triangleq X_T \frac{t}{T} + X_0 \frac{T-t}{T}$ and $\overline{X}(t) \triangleq X(t) - X_r(t)$. The difference $\overline{X}(t)$ follows the dynamic:
\begin{equation}\label{eq:base_system_linear_delayed_null_mean}
\left\{
    \begin{array}{ll}
    d\overline{X}(t) &= \hspace{1em} \left( A(t) \overline{X}(t) + B(t) U(t - h) + \overline{r}(t)\right) dt \\
    &\hphantom{=}\hspace{1em}+ \sigma(t) dW_t, \\
    \overline{X}(0) &= \hspace{1em} 0 , \\
    U(s) &= \hspace{1em} 0 \quad \text{for } s \in [-h, 0[. 
    \end{array}
    \right.
\end{equation}
where $\overline{r}(t) \triangleq r(t) + \dot{X_r}(t) - A(t)X_r(t)$.
Steering the mean and the covariance of $X$ from $\left( \begin{array}{cc} X_0 \\ 0 \end{array} \right)$ to $\left( \begin{array}{cc} X_T \\ \Sigma_T \end{array} \right)$ is thus equivalent to steering the mean and covariance of $\overline{X}(t)$ from $\left( \begin{array}{cc} 0 \\ 0 \end{array} \right)$ to $\left( \begin{array}{cc} 0 \\ \Sigma_T \end{array} \right)$.
Indeed, adding a deterministic term to the controller compensates for the effect induced by the drift and the change of variables \cite{liu_optimal_2023}. More precisely, we consider
$$
U(t) = U_{feedback}(t) + U_{drift}(t),
$$
with
\begin{align*}
U_{drift}(t) = - &B(t+h)^T \Phi_A(t+h)^T G^{-1} \\
&\times\left( \int_0^T \Phi_A(T,s) \overline{r}(s) ds \right) .
\end{align*}
By leveraging similar computations to the ones detailed in \cite{coron_control_2009}, it can be verified that controlling system \eqref{eq:base_system_linear_delayed_null_mean} as required above boils down to steering the mean and covariance of the reduced system 
\begin{equation}\label{eq:base_system_linear_delayed_1}
\left\{
    \begin{array}{ll}
    dX(t) &= \hspace{1em} \left( A(t) X(t) + B(t) U_{feedback}(t - h) \right) dt \\
    &\qquad + \sigma(t) dW_t, \\
    X(0) &= \hspace{1em} 0 , \\
    U(s) &= \hspace{1em} 0 \quad \text{for } s \in [-h, 0[,
    \end{array}
    \right.
\end{equation}
from $\left( \begin{array}{cc} 0 \\ 0 \end{array} \right)$ to $\left( \begin{array}{cc} 0 \\ \Sigma_T \end{array} \right)$ with a control $U=U_{feedback} \in \mathcal{U}$ of zero mean.
Therefore, from now on, we only focus on this latter problem. Of particular interest is the case where $\Sigma_T$ is the smallest possible.


\section{Covariance steering: open-loop controls}\label{CC}

In this section, we study which state covariances can be reached by the solution of~\eqref{eq:base_system_linear_delayed_1} within a finite time frame. In particular, we show that there exists a lower bound on the state covariance under which \eqref{eq:base_system_linear_delayed_1} can not be steered. We then establish the existence of open-loop controls that achieve state covariances that are arbitrarily close to this lower bound.

\subsection{Extending the Artstein transformation}
We achieve the aforementioned goal by leveraging the Artstein transformation \cite{artstein_linear_1982} to our stochastic setting. 
\begin{definition}[Artstein transform]
Let $X$ be the process solving equation \eqref{eq:base_system_linear_delayed_null_mean}. The Artstein transform of $X$ is the following process, which is adapted to $\mathcal{F}$, 
\begin{equation}\label{eq:def_artstein_transform}
Y(t) \triangleq X(t) + \int_{t-h}^t \Phi_A(t,s+h) B(s+h) U(s) ds. 
\end{equation}
\end{definition}
It can be easily checked that the Artstein transform solves the following non-delayed stochastic equation 
\begin{equation}\label{eq:artstein_system_linear}
\left\{
    \begin{array}{ll}
    dY(t) &= \hspace{1em} \left( A(t) Y(t) + \overline{B} U(t) \right) dt + \sigma(t) dW_t, \\
    Y(0) &= \hspace{1em} 0 ,
    \end{array}
    \right.
\end{equation}
where $\overline{B}(t) \triangleq \Phi_A(t,t+h) B(t+h)$. 
We can now apply known results from non-delayed stochastic control from \cite{yong_stochastic_1999, liu_optimal_2023} to the Artstein transform $Y$ to steer \eqref{eq:artstein_system_linear} as desired. 
However, it is unclear how to compute the covariance of $X$, which we aim to estimate once the covariance of $Y$ is available. For this, we make use of the following:
\begin{lemma}
Let $X$ be the process solving equation \eqref{eq:base_system_linear_delayed_null_mean}, and $Y$ its Artstein transform. Then for all $t \in [h,T]$
\begin{equation}\label{eq:link_artstein_original} X(t) = \Phi_A(t,t-h) Y(t-h) + \int_{t-h}^t \Phi_A(t,s) \sigma(s) dW_s . 
\end{equation}
\end{lemma}
\begin{proof}
Using the analytic formula for linear SDEs \cite{yong_stochastic_1999}, we can establish the following relation between $X(t)$ and $X(t-h)$:
\begin{align*}
    X(t) = & \Phi_A(t,t-h) X(t-h) \\
    & + \int_{t-h}^t \Phi_A(t,s) B(s) U(s-h) ds \\
    & + \int_{t-h}^t \Phi_A(t,s) \sigma(s) dW_s.
\end{align*}
By expressing $X(t-h)$ in terms of $Y(t-h)$, a change of variable in the regular integral yields
\begin{align*}
X & (t) =  \Phi_A(t,t-h) Y(t-h) + \int_{t-h}^t \Phi_A(t,s) \sigma(s) dW_s \\
& - \Phi_A(t,t-h)\int_{t-2h}^{t-h} \Phi_A(t-h,s+h) B(s+h) U(s) ds \\
& + \int_{t-2h}^{t-h} \Phi_A(t,s+h) B(s+h) U(s) ds ,
\end{align*}
from which we infer the conclusion thanks to the identity $\Phi_A(t,t-h)\Phi_A(t-h,s+h) = \Phi_A(t,s+h)$. 
\end{proof}
Equation \eqref{eq:link_artstein_original} provides important insights for the controllability of the process $X$. In particular, it states $X(t)$ can only be controlled through $Y(t-h)$, in that the additional noise term $\int_{t-h}^t \Phi_A(t,s) \sigma(s) dW_s$ can not be controlled. 
Let $\Sigma_Y(t) \triangleq \espE[ (Y(t) - \espE[Y(t)]) (Y(t) - \espE[Y(t)])^T ]$ denote the covariance of the process $Y$. Equation \eqref{eq:link_artstein_original} can be further manipulated to express $\Sigma_X$ as a function of $\Sigma_Y$ as explained in the following lemma.
\begin{lemma}\label{lem:relation_covariance_artstein_and_original}
Let $X$ be the process solving equation \eqref{eq:base_system_linear_delayed_null_mean}, and let $Y$ be its Artstein transform. For all $t \in [h,T]$,
\begin{equation}\label{eq:eq_diff_artstein_transform}
\begin{split}
\Sigma_X(t) = &\Phi_A(t,t-h) \Sigma_Y(t-h) \Phi_A(t,t-h)^T \\
&+ \Sigma_{\min}(t) ,
\end{split}
\end{equation}
where
$$
\Sigma_{\min}(t) \triangleq \int_{t-h}^t \Phi_A(t,s) \sigma(s) \sigma(s)^T \Phi_A(t,s)^T ds .
$$
\end{lemma}
\begin{proof}
Fix $t \in [h,T]$.  We first utilize a standard result from stochastic calculus that states that, if $\gamma_1$ and $\gamma_2$ two deterministic functions in $L^2([0,T], \mathbb{R})$, then :
\begin{equation}\label{eq:quadratic_variation_of_stochastic_integral}
\begin{split}
\espE \left[  \left( \int_0^t \gamma_1(s) dW_s  \right) \left( \int_0^t \gamma_2(s) dW_s  \right) \right] = \int_0^t \gamma_1(s) \gamma_2(s) ds
\end{split}
\end{equation}
From this result, we have that
\begin{equation*}
\begin{split}
\espE & \left[  \left( \int_0^t \Phi_A(t,s) \sigma(s) dW_s \right) \left( \int_0^t \Phi_A(t,s) \sigma(s) dW_s  \right)^T \right] \\
& =  \int_0^t \Phi_A(t,s) \sigma(s) \sigma(s)^T \Phi_A(t,s)^T ds = \Sigma_{min}(t)
\end{split}
\end{equation*}
Since the covariance of the sum of two independent random variables is the sum of their covariances, what is left to prove is that $\Phi_A(t,t-h) Y(t-h)$ is independent of $\int_{t-h}^t \Phi_A(t,s) \sigma(s) dW_s$. 
Since $X$ is adapted to $\mathcal{F}$, by definition $Y$ is also adapted to $\mathcal{F}$, and in particular $Y(t-h)$ is $\sigma( W_s: 0 \le s \le t-h )  $-measurable. To conclude, it is sufficient to prove that $\int_{t-h}^t \Phi_A(t,s) \sigma(s) dW_s$ is independent from $W_r$, for $r$ in $[0,t-h]$. Now, if $t-h = s_0 < ... < s_N = t$ is a sequence of partitions of $[t-h,t]$ whose size tends to zero, thanks to the properties of the Îto integral and Assumption 1 we have that
\begin{align*}
    &\int_{t-h}^t \Phi_A(t,s) \sigma(s) dW_s = \\
    &\lim_{N \rightarrow \infty} \sum^{N-1}_{i=0} \Phi_A(t,s_i) \sigma(s_i) (W_{s_{i+1}} - W_{s_i}) , \ \textnormal{in} \ L^2 .
\end{align*}
For every $r \in [0,t-h]$, it thus follows that
\begin{align*}
    &\mathbb{E}\left[ W_r \int_{t-h}^t \Phi_A(t,s) \sigma(s) dW_s \right] = \\
    &= \lim_{N \rightarrow \infty} \sum^{N-1}_{i=0} \Phi_A(t,s_i) \sigma(s_i) \mathbb{E}\left[ W_r (W_{s_{i+1}} - W_{s_i}) \right] = 0 ,
\end{align*}
due to the independence properties of Wiener processes. This concludes the proof.
\end{proof}
The covariance of the state $X$ comprises two components: one determined by the covariance of $Y$, which is under our complete control, and another determined by the covariance of $\int_{t-h}^t \Phi_A(t,s) \sigma(s) dW_s$ which remains unaffected by the control term. Consequently, the state covariance is inherently greater than the latter covariance, presenting a limitation as there is no means to reduce it.

\subsection{Controllability of the Artstein transform}
Our goal is now to minimize the covariance as much as possible and propose control strategies that can approach the lower covariance bound as closely as wanted.
The following result is classical \cite[Theorem 13]{mahmudov_controllability_2001}: 

\begin{theorem}[Controllability of non-delayed SDEs]\label{thm:controllability_stochastic_artstein}
If the Grammian $\overline{G}_\tau^T$ associated to the deterministic part of equation \eqref{eq:artstein_system_linear}, 
defined as
\begin{equation}\label{eq:grammian_artstein} 
\overline{G}_\tau^T \triangleq \int_\tau^T \Phi_A(T,s) \overline{B}(s) \overline{B}(s)^T \Phi_A(T,s)^T ds
\end{equation}
is invertible for all $\tau$ in $[0,T]$, then for every $\Sigma_T \in \mathcal{S}^{++}_n$ there exists $U \in \mathcal{U}$ such that the solution $Y$ to \eqref{eq:artstein_system_linear} associated with $U$ is such that $\mathbb{E}[Y(T-h)] = 0$ and $\Sigma_Y(T-h) = \Sigma_T$.
\end{theorem}

We may combine Lemma \ref{lem:relation_covariance_artstein_and_original} with Theorem \ref{thm:controllability_stochastic_artstein} to infer open-loop controllability of the original control system \eqref{eq:base_system_linear_delayed_1}.

\begin{theorem}[Controllability of delayed SDEs]\label{thm:controllability_stochastic_delayed}
Let $\Sigma_T \in \mathcal{S}^{++}_n$. Under Assumption 1, there exists $U \in \mathcal{U}$ such that the solution $X$ to \eqref{eq:base_system_linear_delayed_null_mean} associated with $U$ verifies $\mathbb{E}[X(T)] = 0$ and $\Sigma_X(T) = \Sigma_T + \Sigma_{\min}(T)$.
\end{theorem}
\begin{proof}
To apply Theorem \ref{thm:controllability_stochastic_artstein}, we need to prove that Assumption 1 guarantees the invertibility of the modified Grammian $\overline{G}_\tau^T$. For this, we may compute
\begin{equation*}
\begin{split}
\overline{G}_\tau^T & = \int_\tau^T \Phi_A(T,s) \overline{B}(s) \overline{B}(s)^T \Phi_A(T,s)^T ds \\
& = \int_\tau^T \Phi_A(T,s) \Phi_A(s,s+h) B(s+h) B(s+h)^T\\
&\hphantom{=}( \Phi_A(T,s) \Phi_A(s,s+h) )^T ds \\
& = \int_\tau^T \Phi_A(T,s+h) B(s+h) B(s+h)^T \Phi_A(T,s+h)^T ds \\
& = \Phi_A(T,T+h) G_{\tau+h}^{T+h} \Phi_A(T,T+h)^T,
\end{split}
\end{equation*}
which is invertible thanks to Assumption 1. Therefore, there exists $U \in \mathcal{U}$ such that the solution $Y$ to \eqref{eq:artstein_system_linear} associated with $U$ is such that $\mathbb{E}[Y(T)] = 0$ and $\Sigma_Y(T-h) = \Phi_A(t,t-h)^{-1} \Sigma_T \Phi_A(t,t-h)^{-T}$, and we conclude.
\end{proof}
Thanks to Theorem \ref{thm:controllability_stochastic_delayed}, we can steer the state $X$ to a final covariance that can be arbitrarily close to the limit $\Sigma_{\min}(T)$, by selecting $\Sigma_T$ arbitrarily small.
The control derived in this section to guide the system is an open-loop stochastic control, posing challenges for numerical implementation \cite{touzi_introduction_2013}.  Consequently, we need to explore the possibility of addressing the steering problem using a feedback controller.

\section{Covariance steering: close-loop controls}\label{CS}

In the previous section, we solved the steering problem by means of open-loop controls. Such control laws are known to be very expensive to compute. To fix this hindrance, we show in this section, the existence of tractable feedback controls that solve the steering problem. We achieve this result by leveraging and combining recent advances in covariance steering techniques, e.g., \cite{ chen_optimal_2016, chen_optimal_2016-1 ,liu_optimal_2023}, with the previous Artstein transformation. In particular, the covariance of solutions to equation~\eqref{eq:base_system_linear_delayed_1} can be steered to $\Sigma_T + \Sigma_{\min}$, with $\Sigma_T \in \mathcal{S}_n^{++}$, via feedback controls. 

In \cite{liu_optimal_2023}, the authors prove that under the controllability conditions of Theorem \ref{thm:controllability_stochastic_artstein}, there exists a unique feedback control $U \in \mathcal{U}$ that steers the covariance of $Y$ solution to \eqref{eq:artstein_system_linear} to a desired covariance, while minimizing the functional cost
$$
J(U) \triangleq \espE \left[ \int_0^T U(t)^T R U(t) dt \right] ,
$$
where $R \in \mathcal{S}^{++}_m$. 
We can show the existence of tractable feedback controls that solve the steering problem by combining \cite{liu_optimal_2023} with our previous results as follows:
\begin{theorem}\label{thm:covariance_steering_artstein}
Let $\Sigma_T \in \mathcal{S}^{++}_n$ and $Y$ denote solutions to \eqref{eq:artstein_system_linear}. Under Assumption 1, there exists a unique feedback control $U^*$ of the form 
\begin{equation}\label{eq:feedback_controller_artstein}
U^*(t) = -R(t)^{-1} \overline{B}(t)^T \Pi(t) Y(t), 
\end{equation}
that minimizes the functional cost $J$ under the constraint $\Sigma_Y(T-h) = \Sigma_T$. Moreover the gain $\Pi(t)$ is the solution to the following Riccati-type ordinary differential equation
\begin{equation}\label{eq:riccati_ODE_covariance_steering}
\left\{
    \begin{array}{ll}
    \dot{\Pi} & = -A^T \Pi - \Pi A + \Pi B R^{-1} B^T \Pi\\
    \Pi(0) &= \Pi_0,
    \end{array}
    \right.
\end{equation}
where $\Pi_0$ is the unique solution of the system 
\begin{equation}\label{eq:system_initial_value_riccati_steering}
\Sigma_T = f(\Pi_0),
\end{equation}
where $f$ is a homeomorphism\footnote{For conciseness, we do not provide the full formula of $f$, see \cite{liu_optimal_2023} for more details.}.
\end{theorem}
\begin{proof}
   The proof of this theorem is a direct consequence of \cite[Lemma 6]{liu_optimal_2023}, whose assumptions are satisfied thanks to Theorem \ref{thm:controllability_stochastic_delayed}.
\end{proof}

Similarly to what has been done in the previous section, once we can control the state covariance of the process $Y$, we can leverage the relation given by lemma \ref{lem:relation_covariance_artstein_and_original} and steer the state covariance of $X$ to any symmetric positive definite matrix above the lower bound $\Sigma_M$. 

This method still presents some limitations, mainly: 
\begin{itemize}
    \item Computing the controller requires solving equation \eqref{eq:system_initial_value_riccati_steering}, which, although it can be done through root-finding algorithms, can be expensive as the function $f$ is not easy to compute.
    \item There are no guarantees on the value of the covariance along the trajectory.
\end{itemize}
In the next section, we therefore introduce an optimal control-based approach to minimizing the covariance along the trajectory.

\section{Global-in-time covariance steering via close-loop controls}\label{OC}


In the previous section, we showed the existence of tractable feedback controls that solve the steering problem. Such control laws may not guarantee that the covariance remains small during the whole control horizon $[0,T]$. In this section, our objective is to establish an efficiently implementable control strategy that keeps the state covariance consistently low along its trajectory, and not just at the final time.
For this, we propose to compute controllers that enable approaching any neighbor of the the covariance threshold via techniques from Linear Quadratic (LQ) control \cite{yong_stochastic_1999}. The goal consists of minimizing the quadratic functional cost $J_R$, 
defined as
\begin{equation}\label{eq:def_quadratic_optimal_cost_2}
\begin{split}
    J_R(U,X) \triangleq & \espE \left[ \int_h^T X(t)^T Q(t) X(t) dt + X(T)^T G X(T) \right] \\
    & + \espE \left[\int_0^{T-h} U(t)^T R(t) U(t) dt \right],
\end{split}
\end{equation}
where $Q \in L^\infty([0,T], \mathcal{S}_n^+)$, $G \in \mathcal{S}_n^+$ and $R \in L^\infty([0,T], \mathcal{S}_n^{++})$. This functional cost penalizes the covariance along the trajectory, the final covariance, as well as the control effort.
Our problem thus states:
\begin{equation}\label{eq:problem_min_cost_X}
    \min_{U \in \mathcal{U}} J_R(U,X) , \quad \text{$X$ solves SDE \eqref{eq:base_system_linear_delayed_null_mean}.}
\end{equation}
To effectively solve \eqref{eq:problem_min_cost_X}, using equation \eqref{eq:link_artstein_original} we replace the process $X$ with the process $Y$ in $J_R$. The benefit of this transformation is that optimal control techniques for non-delayed systems may be leveraged, e.g., \cite{yong_stochastic_1999}.
Under some additional assumptions on the dynamics, we claim that, when the penalization on the control effort via $R$ approaches zero, the controls solutions to \eqref{eq:problem_min_cost_X} enable the covariance of the delayed SDE to track any neighbor of the threshold covariance.

\subsection{Optimal control for delayed SDEs.}
We may change variable in \eqref{eq:problem_min_cost_X} as follows. For $t > h$, let us denote $V_{min,Q}(t)$ the following positive real number: 
$$
V_{min,Q}(t) \triangleq \int_{t-h}^t \sigma(s)^T \Phi_A(t,s)^T Q(t) \Phi_A(t,s) \sigma(s) ds.
$$
\begin{lemma}\label{lem:rewrite_cost_artstein_nondelayed}
The cost function $J_R$ can be written in terms of the state $Y$ as follows:
\begin{equation}\label{eq:def_quadratic_optimal_cost_artstein}
\begin{split}
J_R(U,X) & = \espE \left[ \int_0^{T-h} Y(t)^T \overline{Q}(t) Y(t) dt \right] \\
& + \espE \left[ Y(T-h)^T \overline{G} Y(T-h) \right] \\
& + \espE \left[ \int_0^{T-h} U(t)^T R(t) U(t) dt \right] \\
& + \int_h^T V_{min,Q}(t) dt + V_{min,G}(T),
\end{split}
\end{equation}
where $$\overline{Q}(t) \triangleq \Phi_A(t+h,t)^T Q(t+h) \Phi_A(t+h,t) $$ and $$\overline{G} \triangleq \Phi_A(T+h,T)^T G \Phi_A(T+h,T) $$ are symmetric positive matrices.
\end{lemma}
\begin{proof}
 Since the processes $\Phi_A(t,t-h) Y(t-h)$ and $\int_{t-h}^t \Phi_A(t,s) \sigma(s) dW_s$ are independent, similar computations to the ones in the proof of Lemma \ref{lem:relation_covariance_artstein_and_original} yield, for $t>h$
\begin{align*}
& \espE \left[ X(t)^T Q(t) X(t) \right] = \\
& \espE \left[ Y(t-h)^T \Phi_A(t,t-h)^T Q(t) \Phi_A(t,t-h) Y(t-h) \right] \\
& + \espE \left[ \left( \int_{t-h}^{t} \sigma(s) \Phi_A(t,s) dW_s \right)^T Q(t) \cdot \right. \\
& \left. \hspace{5em} \left( \int_{t-h}^{t} \sigma(s) \Phi_A(t,s) dW_s \right) \right].
\end{align*}
Therefore, by leveraging \eqref{eq:quadratic_variation_of_stochastic_integral} we may compute
\begin{align*}
\espE \left[ X(t)^T Q(t) X(t) \right] & = \espE \left[ Y(t-h)^T \overline{Q}(t-h) Y(t-h) \right] \\
& + V_{min,Q}(t).
\end{align*}
Replacing the latter formula in \eqref{eq:def_quadratic_optimal_cost_2} 
finally yields \eqref{eq:def_quadratic_optimal_cost_artstein}.
\end{proof}
At this step, we can apply the results from \cite[Chapter 6.2]{yong_stochastic_1999} to obtain optimal controls that minimize \eqref{eq:def_quadratic_optimal_cost_2} in the form of feedback controls:
\begin{theorem}\label{thm:analytic_solution_}
The control $U^*$ solution to\eqref{eq:problem_min_cost_X} is given by 
\begin{equation*}
    U^* = - R(t)^{-1} \overline{B}(t)^T P(t) Y(t) ,
\end{equation*}
with $P(\cdot)$ solution to the deterministic Riccati equation
\begin{equation}\label{eq:Riccati_ODE_optimal}
    \left\{
    \begin{array}{ll}
    \dot{P} & = -A^T P - P A - \overline{Q} + P \overline{B}  R^{-1} \overline{B} ^T P\\
    P(0) &= \overline{G}.
    \end{array}
    \right.
\end{equation}
\end{theorem}
\begin{proof}
Since $V_{min,Q}(t)$ is independent from the control variable, Lemma \ref{lem:rewrite_cost_artstein_nondelayed} shows that the minimization problem \eqref{eq:problem_min_cost_X} is equivalent to the following:
\begin{equation}\label{eq:problem_min_cost_Y}
    \min_{u \in \mathcal{U}} \overline{J}_R(U,Y) , \quad \text{$X$ solves SDE \eqref{eq:artstein_system_linear}.}
\end{equation}
where $\overline{J}_R$ is defined as
\begin{equation}\label{eq:def_quadratic_optimal_cost}
\begin{split}
    \overline{J}_R(U,Y) \triangleq & \espE \left[ \int_0^{T-h} Y(t)^T \overline{Q}(t) Y(t) dt  \right] \\
    & + \espE \left[ Y(T-h)^T \overline{G} Y(T-h)\right] \\
    & + \espE \left[ \int_0^{T-h} U(t)^T R(t) U(t) dt \right] . 
\end{split}
\end{equation}
This problem is LQ, and thus we may apply the results in \cite[Chapter 6.2]{yong_stochastic_1999} to conclude.
\end{proof}

Thanks to Lemma \ref{lem:rewrite_cost_artstein_nondelayed}, we may note that \eqref{eq:def_quadratic_optimal_cost_2} is lower-bounded by the following quantity:
\begin{equation} \label{eq:minVarTrack}
    V_{min} \triangleq \int_h^T V_{min,Q}(t) dt + V_{min,G}(T),
\end{equation}
We will prove that, under some additional assumptions, 
when $R$ tends to zero it holds that
$$
\min_{U \in \mathcal{U}} J_R(U,X) \underset{R \rightarrow 0}{\rightarrow} V_{min} ,
$$
meeting our desired goal.

\subsection{Convergence equivalence}


From now on, we select the specific control weight $R(t) = \rho I$, where $\rho > 0$ is a user-defined parameter that characterizes the control effort in \eqref{eq:def_quadratic_optimal_cost_2}. We accordingly denote the associated cost functional $J_R$ by $J_\rho$, which corresponds to
\begin{equation}\label{eq:def_quadratic_optimal_cost_3}
\begin{split}
    J_\rho(U,X) \triangleq & \espE \left[ \int_h^T X(t)^T Q(t) X(t) dt + X(T)^T G X(T) \right] \\
    & + \rho \ \espE \left[\int_0^{T-h} U(t)^T U(t) dt \right] .
\end{split}
\end{equation}
We denote by $U_\rho$ the corresponding optimal control, whose existence is granted by Theorem \ref{thm:analytic_solution_}. Finally, we denote $X_\rho$ the corresponding state trajectory.
Our objective is to give some conditions on the dynamics under which the weighted variance of $X_\rho$ approaches the threshold quantity \eqref{eq:minVarTrack} as $\rho$ approaches zero.
The following lemma plays a crucial role in achieving this goal:

\begin{lemma}\label{lem:equivalence_convergence_controls}
Let $U_n$ be a sequence of controls in $\mathcal{U}$ such that, for any matrix $S$ in $\mathcal{S}_n^+$, the associated state $Y_n$ obtained through equation \eqref{eq:artstein_system_linear} verifies
\begin{equation}\label{eq:convergence_Xn_assumption}
\forall t \in [0,T-h], \quad     \espE \left[ Y_n(t)^T S Y_n(t) \right] \underset{n \rightarrow \infty}{ \rightarrow} 0.
\end{equation}
Then,
\begin{equation}\label{eq:convergence_optimal_control}
J_\rho( U_\rho ) \underset{\rho \rightarrow 0}{ \rightarrow} \int_h^T V_{min,Q}(t) dt + V_{min,G}(T) .
\end{equation}
\end{lemma}
\begin{proof}
    Let us first note that the function $\rho \mapsto J_\rho( U_\rho )$ is an increasing function of $\rho$. Indeed, if $0 < \rho_1 < \rho_2$ are two positive real numbers, by optimality it holds that
    $$
    J_{\rho_2}( U_{\rho_2}) \geq J_{\rho_1}( U_{\rho_2}) \geq J_{\rho_1}( U_{\rho_1}) .
    $$
    Let us select a sequence $(\rho_n)_{n \in \mathbb{N}} \subseteq (0,\infty)$ such that
    $$
    \left\{
        \begin{array}{ll}
        & \rho_n \espE \left[ \int_0^{T-h} U_n(t)^T U_n(t) \; dt \right] \underset{n \rightarrow \infty}{ \rightarrow} 0\\
        & \rho_n \underset{n \rightarrow \infty}{ \rightarrow} 0.
        \end{array}
    \right.
    $$
    Due to  \eqref{eq:def_quadratic_optimal_cost_artstein}, we have that 
    \begin{align*}
    J_{\rho_n}(U_n) & = \espE \left[ \int_0^{T-h} Y_n(t)^T \overline{Q}(t) Y_n(t) dt \right] \\
    & + \espE \left[ Y_n(T-h)^T \overline{G} Y_n(T-h) \right] \\
    & + \rho_n \espE \left[ \int_0^{T-h} U_n(t)^T U_n(t) dt \right] \\
    & + \int_h^T V_{min,Q}(t) dt + V_{min,G}(T).
    \end{align*}
    By definition of $\rho_n$ and using the assumptions made on $Y_n$, we obtain
    $$
    J_{\rho_n}(U_n) \underset{n \rightarrow \infty}{ \rightarrow} \int_h^T V_{min,Q}(t) dt + V_{min,G}(T).
    $$
    The optimality of $U_{\rho_n}$ yields
    $$
    \int_h^T V_{min,Q}(t) dt + V_{min,G}(T) \leq J_{\rho_n}(U_{\rho_n}) \leq J_{\rho_n}(U_n),
    $$
    which implies that 
    $$
    J_{\rho_n}(U_{\rho_n}) \underset{n \rightarrow \infty}{ \rightarrow} \int_h^T V_{min,Q}(t) dt + V_{min,G}(T).
    $$
    To conclude the proof, note that since $\rho \mapsto J_\rho( U_\rho )$ is an increasing function of $\rho$, then the above limit is unique for any sequence $(\rho_n)_{n \in \mathbb{N}} \subseteq (0,\infty)$ that converges to zero. 
\end{proof}

Establishing the convergence $\eqref{eq:convergence_optimal_control}$ consists of proving the existence of a sequence of controls that drives the state variance toward the minimum threshold. However, achieving this existence is not straightforward, as larger controllers may force the integral of the weighted variance to divergence. Below, we provide examples of SDEs for which finding such a sequence of controls is feasible.

\subsection{Sufficient conditions for convergence equivalence}


\subsubsection{Pole placement with a normal matrix}

Let us consider the following assumption:
\begin{assumption}\label{asm:assumptions_normal_autonomous_system} The followings hold true:

\begin{enumerate}
    \item The matrices $A$ and $B$ are time independent.
    \item There exists a sequence of matrices $K_n \in \mathbb{R}^{m \times n}$ such that every $H_n \triangleq A - \overline{B} K_n$ is normal, and
    \begin{equation}\label{eq:assumption_normal_max_eigenvalue}
        \max( \Re(Sp( H_n ) )) \leq - n .
    \end{equation}
\end{enumerate}
\end{assumption}


\begin{lemma}\label{lem:estimate_for_normal_matrix}
    Let $S \in \mathcal{S}_n^+$. Under Assumption \ref{asm:assumptions_normal_autonomous_system}, the sequence of controls $U_n(t) = K_n Y_n(t)$ satisfies 
    \begin{equation}\label{eq:estimate_control_autonomous_system_normal}
    \espE \left[ Y_n(t)^T S Y_n(t) \right] \leq  \frac{\Vert S \Vert \Vert \sigma \Vert_\infty^2}{2 n}.
    \end{equation}
\end{lemma}
\begin{proof}
The solution $Y_n$ to \eqref{eq:artstein_system_linear} with feedback control $U_n$ can be explicitly computed as
\begin{align*}
Y_n(t) = \int_0^t e^{H_n(t-s)} \sigma(s) dW_s .
\end{align*}
By applying Ito's quadratic variation formula to $Y_n$, we may therefore compute
\begin{align*}
& \espE \left[ Y_n(t)^T S Y_n(t) \right]  =\\
& \int_0^{t-h} \sigma(s)^T e^{H^T_n(t-h-s)}S e^{H_n(t-h-s)} \sigma(s) ds \\
& = \int_0^{t-h} \Vert \sqrt{S} e^{H_n(t-h-s)} \sigma(s) \Vert^2 ds \\
& \leq \Vert S \Vert\int_0^{t-h} \Vert e^{H_n(t-h-s)} \sigma(s) \Vert^2 ds.
\end{align*}
Since $H_n$ is normal and $(t-h-s) > 0$, it holds that
$$\Vert e^{H_n(t-h-s)} \sigma(s) \Vert^2 \leq \max(Sp(e^{H_n(t-h-s)}))^2 \Vert \sigma(s) \Vert^2 ,$$
from which we finally infer that
\begin{align*}
& \espE \left[ Y_n(t)^T S Y_n(t) \right] \le \\
& \leq \Vert S \Vert \Vert \sigma \Vert_\infty^2 \int_0^{t-h} e^{-2n(t-h-s)}   ds \leq \frac{\Vert S \Vert \Vert \sigma \Vert_\infty^2}{2 n}.
\end{align*}
\end{proof}

Characterizing systems that satisfy Assumption 11 can be difficult. Yet, \eqref{eq:estimate_control_autonomous_system_normal} is in particular satisfied by the fairly large class of fully actuated systems, see the next section. Note that assuming fully actuation is not particularly restricting, in that, if this assumption is not satisfied, one can prove the existence of random states that can not be reached by any open-loop control, see, e.g., \cite{wang_exact_2017}.

\vspace{5px}

\subsubsection{Fully actuated systems}

\begin{assumption}\label{asm:assumptions_overactuated_system} $B(t)$, $t \in [0,T-h]$, has rank $n$.
\end{assumption}


\begin{lemma}
     Let $S \in \mathcal{S}_n^+$. Under Assumption  \ref{asm:assumptions_overactuated_system}, there exists a matrix $K(t) \in \mathbb{R}^{m \times n}$ such that 
    $$
    \overline{B}(t) K(t) = I.
    $$
    In particular, the control sequence defined by
    $$
    U_n(t) = - K(t) (A(t) + n I) Y_n(t)
    $$
    yields the estimate \eqref{eq:estimate_control_autonomous_system_normal}.
\end{lemma}
\begin{proof}
    Thanks to Assumptions \ref{asm:assumptions_on_the_dynamic} and \ref{asm:assumptions_overactuated_system}, $\overline{B}(t)$ has rank $n$ for all $t \in [0,T-h]$, and therefore it admits a stable right inverse. Each process $Y_n$, stemming from each feedback control $U_n$ as defined above, satisfies
    \begin{align*}
    dY_n(t) & = ( A(t) Y_n(t) + \overline{B}(t) U_n(t) )dt + \sigma(t) dW_t \\
    & = - n Y_n(t) +  \sigma(t) dW_t ,
    \end{align*}
    and we can therefore follow the same arguments as in the proof of Lemma \ref{lem:estimate_for_normal_matrix} to obtain the estimate \eqref{eq:estimate_control_autonomous_system_normal}.
\end{proof}

Summing up, thanks to Theorem \ref{thm:analytic_solution_}, we come up with a closed-loop-control-based method to steer the delayed stochastic system, while keeping the state variance close to the threshold variance at will. We demonstrate the efficiency of this method via numerical simulations next. 

\section{Numerical results}\label{NR}

To validate our theoretical findings, we implemented our optimal control-based method to regulate the temperature of a building in realistic settings, where the heat source deliver thermal energy up to some delay. 
We leveraged and enhanced the dynamical models outlined in \cite{van_der_zwan_operational_2020}. In particular, we made these models more realistic by adding noise. This originates from various sources: 1) the stochastic nature of external temperature fluctuations (modeled using an Ornstein-Uhlenbeck process, see, e.g., \cite{prabakaran_temperature_2020}), and 2) the unpredictable usage of the building (modeled via the coefficient $\sigma_i$ below). Then, the dynamics are given by
\begin{align*}
\begin{array}{ll}  dT(t) = \left[ \left( \begin{array}{cc} - R_e & R_e \\ 0 & - \theta \end{array} \right) T(t) + \left( \begin{array}{c} R_u \\ 0 \end{array} \right) U(t-h) \right] dt \\
+ \left( \begin{array}{c} - R_e T_{eq} \\ \theta T_p(t) + \dot{T}_p(t) \end{array} \right)  dt  + \left( \begin{array}{c} \sigma_i \\ \sigma_e \end{array} \right) dW_t , \end{array}
\end{align*}
where $R_e = 5 \ 10^{-4}, R_u = 2 \ 10^{-4}, \theta = 3.5 \ 10^{-4}, T_{eq} = 20, \sigma_e = 0.1, \sigma_i = 0.05$, and with baseline (i.e., ``predicted'') external temperature $T_p(t) = 5+5 \cos(0.004 t)$. 
These dynamic are two dimensional. The first variable models the evolution of the temperature of the building (to control), while the second one models the evolution of the external temperature (that we can not control), which has mean $T_p(t)$ and quadratic variation $\int_0^t e^{- 2 \theta(t-s)} \sigma(s)^2 ds$. 
The simulation horizon is five days. This enables to stress test our control strategy in mitigating the fluctuations of the external temperature, that are induced by the day-night cycle. 
The threshold variance of the temperature of the building amounts to $V_{min} = 1.76$.

Figure~\ref{fig:simu_bat_temp_double_bruit} shows the trajectory of the system using our optimal control-based feedback. 
The temperature is efficiently stabilized even under poor knowledge of the external temperature. Figure \ref{fig:simu_bat_var_double_bruit} shows the evolution of the variance of the temperature of the building. Our method enables to successfully track the threshold variance under controls with high gains, as granted by Theorem \ref{thm:controllability_stochastic_delayed}.

\begin{figure}[ht!]
\centering
  \includegraphics[width=1.0\linewidth]{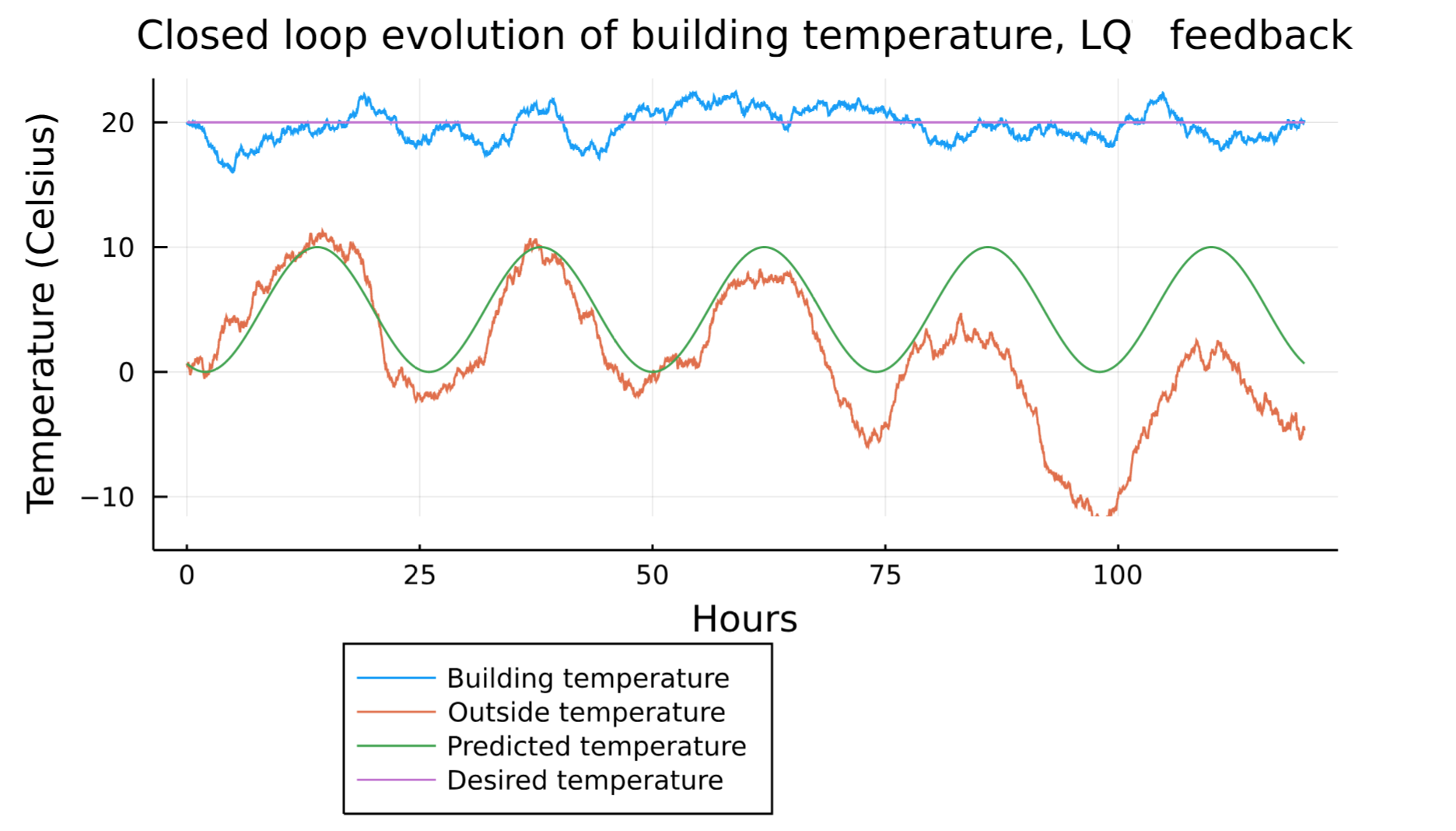}
  \caption{\centering A trajectory with an optimal control-based feedback with high gain.}
  \label{fig:simu_bat_temp_double_bruit}
\end{figure}

\begin{figure}[ht!]
  \centering
  \includegraphics[width=1.0\linewidth]{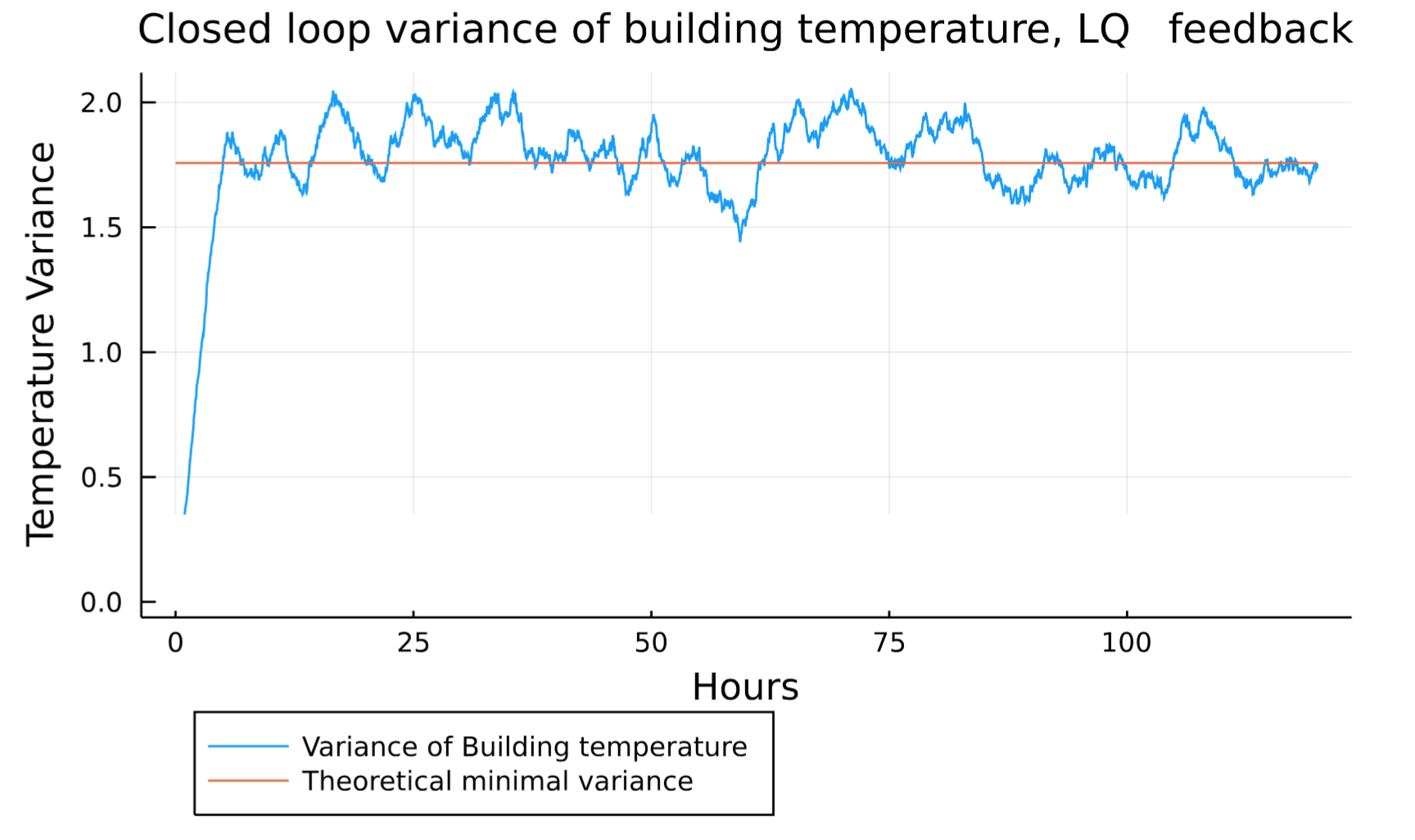}
  \caption{\centering State variance through time compared to the theoretical lower bound.}
  \label{fig:simu_bat_var_double_bruit}
\end{figure}

\section{Conclusion}\label{CL}

In this paper, we solved the (mean-covariance) steering problem for a general class of linear SDEs subject to an input delay. Our objective consisted of steering in finite time the state of the system to a final state with some given final probability distribution. In particular, we proposed a method to minimize the final state covariance. 
We did so by leveraging the Artstein transformation, thanks to which we derived a "predictor" of the state that tracks a non-delayed SDE. We then established a linear relationship between the predictor covariance and the original state covariance. In particular, this relationship revealed the existence of an structural minimal covariance below which the system can not be steered, and which is essentially due to the presence of the delay. Nevertheless, we proved the system can be steered to any covariance that is greater than this minimal covariance. 
Finally, we introduced an optimal control-based approach to minimize the system covariance throughout the whole control horizon, computing upper bounds for this minimal variance. We assessed the efficiency of this control strategy via numerical simulations on realistic stochastic systems affected by delays. 

Several exciting research directions 
are listed hereafter. Extending our approach to linear SDEs with multiplicative noise would enable the modeling of more sophisticated, though relevant systems. However, due to the possible lack of the Artstein transformation, it could be challenging to derive an explicit expression for the variance threshold as we did in the present work. Additionally, it would be interesting to explore systems with multiple input delays. One could build upon previous studies on optimal control of ODEs with multiple input delays \cite{basin_optimal_2006}, and then try to extend suchn results to more general SDEs. 

\bibliographystyle{ieeetr}
\bibliography{references}

\end{document}